\documentclass[12pt]{article}
\usepackage{amsfonts}
\usepackage{amssymb, amscd, amsthm}
\usepackage{latexsym}
\usepackage{multirow}
\usepackage{amsmath}
\usepackage[all]{xy}
\usepackage[dvips]{graphicx}
\usepackage{mathrsfs}
\usepackage{fancyhdr}

\newtheorem{lemma}[subsection]{Lemma}
\newtheorem{thm}[subsection]{Theorem}
\newtheorem{prop}[subsection]{Proposition}
\newtheorem{rem}[subsection]{Remark}
\newtheorem{coro}[subsection]{Corollary}

\newtheorem{con}[subsection]{Conjecture}

\newcommand{\ra}{\rightarrow}
\newcommand{\mdr}{M(d,r)}
\newcommand{\mone}{\mo_{\p^2}(1)}
\newcommand{\mmon}{\mo_{\p^2}(-1)}
\newcommand{\mo}{\mathcal{O}}

\newcommand{\p}{\mathbb{P}}
\newcommand{\bt}{\mathbb{T}}
\newcommand{\bc}{\mathbb{C}}
\newcommand{\bl}{\mathbb{L}}

\begin{document}
\fontsize{12pt}{14pt} \textwidth=14cm \textheight=21 cm
\numberwithin{equation}{section}
\title{Affine pavings for moduli spaces of pure sheaves on $\mathbb{P}^2$ with degree $\leq 5$.}
\author{Yao YUAN}
\date{\small\textsc 
MSC, Tsinghua University, Beijing 100084, China
\\ yyuan@mail.math.tsinghua.edu.cn.}
\maketitle
\begin{flushleft}{\textbf{Abstract.}}
Let $M(d,r)$ be the moduli space of semistable sheaves of rank 0, Euler characteristic $r$ and first Chern class $dH~(d>0)$, with $H$ the hyperplane class in $\p^2$.  In \cite{myself} we gave an explicit description of the class $[\mdr]$ of $\mdr$ in the Grothendieck ring of varieties for $d\leq 5$ and $g.c.d(d,r)=1$.  In this paper we compute the fixed locus of $\mdr$ under some $(\bc^{*})^2$-action and show that $\mdr$ admits an affine paving for $d\leq 5$ and $g.c.d(d,r)=1$.  We also pose a conjecture that for any $d$ and $r$ coprime to $d$, $\mdr$ would admit an affine paving.   
\end{flushleft}

\section{Introduction.}
There are many interesting results on moduli spaces of 1-dimensional semistable sheaves on surfaces, mainly with surfaces K3, abelian or rational, for instance \cite{jcmm}, \cite{jmdmm}, \cite{lee}, \cite{mmz}, \cite{ky}, \cite{yy}, \cite{yuan} and \cite{myself}.  Recently because of its close relation to PT-invariants (defined in \cite{pt}) on local Calabi-Yau 3-fold, this kind of moduli spaces over Fano surfaces become even more worthy of study.  

Let $M(d,r)$ be the moduli space of semistable sheaves of rank 0, first Chern class $dH~(d>0)$ and Euler characteristic $r$ on $\p^2$.  In Physics, the Euler number $e(M(d,r))$ up to a sign is called the BPS state of weight 0 for local $\p^2$ (see Equation (4.2) and Table 4 in Section 8.3 in \cite{kkv}).  Toda's work in \cite{toda} implies that $e(\mdr)$ only depends on $d$ as long as $r$ coprime to $d$.  However there is so far no good understanding for a general $\mdr$ even at the Euler number level.    

For $d\leq 5$ and $g.c.d.(d,r)=1$, we have already shown that the class $[\mdr]$ of $\mdr$ in the Grothendieck ring of varieties can be expressed as a summation $\sum b_{2i}\bl^i$ with $\bl^i:=[\mathbb{A}^i]$, and we have also computed all the numbers $b_{2i}$ (See Theorem 5.1, Theorem 5.2, Theorem 6.1 in \cite{myself}).  In this paper we study the fixed locus on $\mdr$ of the $(\bc^{*})^{2}$-action induced by some $(\bc^{*})^2$-action on $\p^2$, and we prove the following theorem.
\begin{thm}\label{main}For $d\leq 5$ and $g.c.d.(d,r)=1$, the $(\bc^{*})^2$-fixed locus of $\mdr$ can be decomposed into a union of affine spaces, i.e. $\mdr$ admits an affine paving.  More precisely we have

1, The $(\bc^{*})^2$-fixed loci of $M(1,1)$, $M(2,1)$ and $M(3,1)$ consist of 3, 6 and 27 points respectively.

2, The $(\bc^{*})^2$-fixed locus of $M(4,1)$ can be decomposed into the union of 186 points and 6 affine lines $\mathbb{A}^1$.

3, The $(\bc^{*})^2$-fixed loci of $M(5,1)$ can be decomposed into the union of 1545 points,  144 affine lines $\mathbb{A}^1$ and 6 affine planes $\mathbb{A}^2$.

4, The $(\bc^{*})^2$-fixed loci of $M(5,2)$ can be decomposed into the union of 1506 points,  186 affine lines $\mathbb{A}^1$ and 3 affine planes $\mathbb{A}^2$.
\end{thm}
Theorem \ref{main} implies that the formulas given in \cite{myself} not only provide motive decompositions but also cell decompositions of the moduli spaces.

\begin{rem}\label{otnis}Theorem 6.1 in \cite{myself} shows that $[M(5,1)]=[M(5,2)]$.  But these two moduli spaces don't have the same $(\bc^{*})^2$-fixed locus, hence we know there is no $(\bc^{*})^2$-equivariant isomorphisms between them.  
\end{rem}
We expect Theorem \ref{main} is true in larger generality and we pose a conjecture as follows for the future study.
\begin{con}\label{icon}$\mdr$ admits an affine paving for all $d$ and $g.c.d.(d,r)=1$.
\end{con}
They showed in \cite{jcmm} that the $(\bc^{*})^2$-fixed locus of $M(4,1)$ consists of 180 isolated points and 6 projective lines $\p^1$.  Their result implies ours for $d=4$.  Nevertheless we also include case $M(4,1)$ for the wholeness of the context.  Actually based on the stratification given in \cite{myself}, it is just an easy exercise to find out all the fixed points in $M(4,1)$.  

The structure of the paper is arranged as follows.  In Section 2, we define the $(\bc^{*})^2$-action on $\mdr$ induced by some $(\bc^{*})^2$-action on $\p^2$, then we recall the definition of some big open subset $W^d$ in $M(d,1)$ and show that $W^d$ is invariant under the action with fixed points isolated.  Then we prove Theorem \ref{main} for $d\leq3$, using the fact by Theorem 5.1 in \cite{myself} that $W^d=M(d,1)$ for $d\leq 3$.  In Section 3,  we study the fixed locus of $M(4,1)$.  The last two moduli spaces, $M(5,1)$ and $M(5,2)$, are dealt in the last section, Section 4.  

\begin{flushleft}{\textbf{Acknowledgments.}} I was partially supported by NSFC grant 11301292.  When I wrote this paper, I was a post-doc at MSC in Tsinghua University in Beijing.  Finally I thank Y. Hu for some helpful discussions.
\end{flushleft}      
\section{The $(\bc^{*})^2-$action and fixed points in $W^d$.}
Let $\bt:=(\bc^{*})^2$ acts on $\p^2$ by $(t_1,t_2)\cdot[x,y,z]=[x,t_1^{-1}y,t_2^{-1}z]$ for all $(t_1,t_2)\in\bt$ and $[x,y,z]\in\p^2$.  Denote by $\theta_{(t_1,t_2)}$ the automorphism of $\p^2$ given by the action of the point $(t_1,t_2)$.  The $\bt$-action on $\p^2$ induces an action on $\mdr$ defined by $(t_1,t_2)\cdot F=\theta_{(t_1,t_2)}^{*}F$ for any sheaf $F\in\mdr$.  

We recall some notations we used in \cite{myself}.  At first we say that a pair $(E,f)$ always satisfies the following two conditions. 
\begin{equation}\label{dsl}
(1) E\simeq\bigoplus_i \mo_{\p^2}(n_i)~i.e.E~is~a~direct~sum~of~line~bundles~on~\p^2;
\end{equation}  
\begin{equation}\label{inm}
(2) f\in Hom(E\otimes \mo_{\p^2}(-1),E)~and~moreover~f~is~injective.~~~~~~
\end{equation} 

By Definition 2.1 in \cite{myself}, $(E,f)\simeq(E',f')$ if $E\simeq E'$ and $\exists ~\varphi,\phi\in Isom(E,E')$, s.t. $f'\circ (\varphi\otimes id_{\mo_{\p^2}(-1)})=\phi\circ f$.  

There is 1-1 correspondence between isomorphism classes of pairs $(E,f)$ and isomorphism classes of pure sheaves $F$ of dimension 1 on $\p^2$ by assigning $(E,f)$ to $coker(f)$ (Proposition 2.5 in \cite{myself}).  We can put a stability condition on $(E,f)$ so that $(E,f)$ is (semi)stable iff so is $coker(f)$ (Definition 3.1 in \cite{myself}).
 
If there is no strictly semistable sheaves in $\mdr$, i.e. $r$ is coprime to $d$, then we can assign every point $F$ in $M(d,r)$ uniquely to a pair $(E,f)$ such that $coker(f)\simeq F$.

We view points in $\mdr$ as stable pairs $(E,f)$, then the $\bt$-action on $\mdr$ is quite explicit.  We have $(t_1,t_2)\cdot (E,f)=(\theta_{(t_1,t_2)}^{*}E,\theta_{(t_1,t_2)}^{*}f)\simeq(E,\theta_{(t_1,t_2)}^{*}f)$.  $f$ can be represented by a $d\times d$ matrix with $(i,j)$-th entry $a_{ij}(x,y,z)$ homogenous polynomials of $x,y,z$.  Then $\theta_{(t_1,t_2)}^{*}f$ is a map represented by the matrix with $(i,j)$-th entry $a_{ij}(x,t_1y,t_2z)$.  $F$ is fixed by $\bt$ iff $\forall (t_1,t_2)\in\bt$, $\exists$ an invertible matrix with entries $m_{(t_1,t_2)}^{ij}\in\bigoplus_{n\geq0}H^0(\mo_{\p^2}(n))$, such that $a_{ij}(x,y,z)=\sum_l a_{il}(x,t_1y,t_2z)m^{lj}_{(t_1,t_2)}$ for all $1\leq i,j\leq d$.

We stratify $M(d,r)$ by the form of $E$, then every stratum is a constructible set in $M(d,r)$ and invariant under the $\bt$-action.  

Let $r=1$.  Define $\widetilde{W}^d:=\{(E,f)|E\simeq\mo_{\p^2}\oplus\mmon^{\oplus d-1}\}\subset M(d,1).$  By the stability condition for a pair $(E,f)\in\widetilde{W}^d$, $f$ can be represented by the following matrix
\begin{equation}\label{matrix}\left(\begin{array}{ccc}0&1&0\\ A&0 &B\end{array}\right),\end{equation}
where $A$ is a $(d-1)\times 1$ matrix with entries in $H^0(\mo_{\p^2}(2))$ and $B$ a $(d-1)\times (d-2)$ matrix with entries in $H^0(\mone)$.  

Let $f_{B^t}:\mmon^{\oplus d-2}\ra\mo_{\p^2}^{\oplus d-1}$ be a morphism represented by the transform $B^t$ of $B$.  
Then $f_{B^t}$ is injective with cokernel a rank 1 sheaf $Q_f$. 

The dual $G^{\vee}$ of any 1-dimensional sheaf $G$ is defined as $G^{\vee}:=\mathcal{E}xt^1(G,\mo_{\p^2})$.  If $G$ is pure, then $G^{\vee\vee}\simeq G$ and moreover $G$ and $G^{\vee}$ are determined by each other (see \cite{yuan} Lemma A.0.13).
Any sheaf $F=coker(f)$ in $\widetilde{W}^d$ has its dual $F^{\vee}$ determined by the following sequence (see the diagram (4.2) in \cite{myself}).
\begin{equation}\label{dus}\xymatrix{0\ar[r]&\mo_{\p^2}(-2)\ar[r]^{~~~\sigma_f}&Q_f\ar[r]&F^{\vee}\otimes\mo_{\p^2}(-2)\ar[r]&0.}
\end{equation}

Define $W^d:=\{[(E,f)]\in \widetilde{W}^d|Q_f~is~torsion~free\}\subset\widetilde{W}^d.$  The complement of $W^d$ in $M(d,1)$ is of codimension $\geq2$.  We have 
\begin{prop}\label{bigopen}$W^d$ is $\bt$-invariant with isolated $\bt$-fixed points. 
\end{prop}
\begin{proof}$\widetilde{W}^d$ is invariant under the $\bt$-action.  For every $(t_1,t_2)\in\bt$, $\theta_{(t_1,t_2)}^{*}Q_f$ is torsion free $\Leftrightarrow$ $Q_f$ is torsion free.  Hence $W^d$ is invariant under the $\bt$-action.  

By Lemma 4.6 in \cite{myself}, if $Q_f$ is torsion free, then $Q_f\simeq I_{\bar{d}}\otimes\mo_{\p^2}(d-2)$, with $\bar{d}:=\frac{(d-1)(d-2)}2$ and $I_{\bar{d}}$ some ideal sheaf of $\bar{d}$-points not lying on a curve of class $(d-3)H$ on $\p^2$.  

For $Q_f$ torsion free, $\sigma_f$ in (\ref{dus}) gives an element in $H^0(Q_f\otimes\mo_{\p^2}(2))=H^0(I_{\bar{d}}\otimes\mo_{\p^2}(d))$ which is exactly $det(f)$.  Hence the class $[\sigma_f]\in\p (H^0(I_{\bar{d}}\otimes\mo_{\p^2}(d)))$ is determined by the support $Supp(F)$ of $F$.  

Therefore, $F$ is fixed by $\bt$ iff both $I_{\bar{d}}$ and $Supp(F)$ are fixed by $\bt$.  Hence $W^d$ has finitely many $\bt$-fixed points because so do  the Hilbert scheme $Hilb^{[\bar{d}]}(\p^2)$ of $\bar{d}$-points on $\p^2$ and the linear system $|dH|$.  Hence the proposition.
\end{proof}

Notice that for $d\leq 4$, up to isomorphism $M(d,1)$ is the only moduli space such that there is no strictly semistable locus, since $M(d,r)\simeq M(d,r')$ if $r\equiv\pm r'$ (mod $d$).  We have
\begin{coro}\label{dlth}For $d\leq 3$ and $r$ coprime to $d$, $\mdr$ admits an affine paving.  More precisely, the $\bt$-fixed loci of $M(1,1)$, $M(2,1)$ and $M(3,1)$ consist of 3, 6 and 27 points respectively.
\end{coro}
\begin{proof}By Theorem 5.1 in \cite{myself}, we have $W^d=M(d,1)$ for $d\leq3$, and by direct computation we get the corollary.
\end{proof}

\section{$\bt$-fixed points in $M(4,1)$.}

In this section we prove the following theorem.
\begin{thm}\label{flfour}The $\bt$-fixed locus of $M(4,1)$ can be decomposed into the union of 186 points and 6 affine lines $\mathbb{A}^1$.  Hence it admits an affine paving.
\end{thm}

$M(4,1)$ can be stratified into two following $\bt$-invariant strata.
\begin{eqnarray} & M_1:=&\widetilde{W}^4=\{[(E,f)]\in M(4,1)|E\simeq\mo_{\p^2}\oplus\mmon^{\oplus 3}\};\nonumber\\& M_2:=&\{[(E,f)]\in M(4,1)|E\simeq\mo_{\p^2}^{\oplus 2}\oplus\mmon\oplus\mo_{\p^2}(-2)\}.\nonumber
\end{eqnarray}
For a pair $(E,f)\in M_2$, $f$ can be represented by the following matrix
\begin{equation}\label{romf}\left(\begin{array}{cccc}b_1&b_2&0&0\\0&0&1&0\\0&0&0&1\\a_1&a_2&0&0\end{array}\right),
\end{equation}
where $b_i\in H^0(\mone)$ and $a_i\in H^0(\mo_{\p^2}(3))$.  The stability condition for such $(E,f)$ is equivalent to $kb_1\neq k'b_2$ for any $(k,k')\in\bc^2-\{0\}$.  

\begin{lemma}\label{mtfour}The $\bt$-fixed locus of $M_2$ consists of 42 points.
\end{lemma}
\begin{proof}The proof of Lemma 5.4 in \cite{myself} shows that points in $M_2$ 1-1 correspond to isomorphism classes of pairs $(R_f,[\omega_f])$, with $R_f\simeq I_1$ and $[det(f)]=[\omega_f]\in\p( H^0(R_f\otimes\mo_{\p^2}(4)))$.  Moreover, $\theta_{(t_1,t_2)}^{*}f\simeq f$ iff $\theta_{(t_1,t_2)}^{*}R_f\simeq R_f$ and $[\theta_{(t_1,t_2)}^{*}\omega_f]=[\omega_f]$.  Every $\bt$-fixed point in $M_2$ corresponds to a $\bt$-fixed point on $\p^2$ together with a $\bt$-fixed curve of degree 4 passing it.  Hence the lemma.
\end{proof}

\begin{lemma}\label{wfour}The $\bt$-fixed locus of $W^4$ consists of 120 points.
\end{lemma}
\begin{proof}$\bar{d}=3$ for $d=4$.  There are 10 $\bt$-fixed ideal sheaves $I_3$ of 3-points not lying on a curve in $|H|$.  For each $I_3$ there are 12 $\bt$-fixed points in $\p(H^0(I_3(4)))$.  Hence the lemma. 
\end{proof}

\begin{lemma}\label{tfour}The $\bt$-fixed locus of $M_1-W^4$ consists of 24 points and 6 affine lines $\mathbb{A}^1$.
\end{lemma}  
\begin{proof}Every pair $(E,f)$ in $M_1-W^4$ can be assigned uniquely to a pair $(Q_f,[\sigma_f])$ such that $Q_f$ lies in the following non-split exact sequence
\begin{equation}\label{tqf}\xymatrix{0\ar[r]&\mo_{H}(-1)\ar[r]^{~~\jmath} &Q_f\ar[r]^{p~~~} &\mo_{\p^2}(1)\ar[r] &0,}\end{equation}
where $\mo_H(-1):=\mo_C\otimes\mo_{\p^2}(1)\simeq \mo_{\p^1}(-1)$ for some curve $C\in|H|$.  $Q_f$ is unique up to isomorphism if the support of its torsion $\mo_H(-1)$ is given.  

Take global sections of (\ref{tqf})$\otimes\mo_{\p^2}(2)$ and we get 
\begin{equation}\label{tqfgs}\xymatrix@C=0.7cm{0\ar[r]&H^0(\mo_{H}(1))\ar[r]^{\bar{\jmath}~~~~} &H^0(Q_f\otimes\mo_{\p^2}(2))\ar[r]^{~~~\bar{p}} &H^0(\mo_{\p^2}(3))\ar[r] &0,}\end{equation}
and $[\sigma_f]\in \p(H^0(Q_f\otimes\mo_{\p^2}(2)))$ not contained in $\p(\bar{\jmath}(H^0(\mo_{H}(1))))$. 

$(E,f)$ is $\bt$-fixed $\Leftrightarrow$ $Q_f$ and $[\sigma_f]$ are both $\bt$-fixed $\Leftrightarrow$ the support of $\mo_H(-1)$ in (\ref{tqf}) and $[\sigma_f]$ are both $\bt$-fixed.  

Let $Q_f$ be $\bt$-fixed.  Assume that the torsion of $Q_f$ is $\mo_{\{x=0\}}(-1)$.  Every element $\gamma$ in $H^0(Q_f\otimes\mo_{\p^2}(2))$ can be assigned to a matrix $N_{\gamma}$ of form (\ref{matrix}), more precisely, write $\gamma\sim\underline{a^{\gamma}}:=(a_1^{\gamma},a_2^{\gamma},a_3^{\gamma})$ with $a^{\gamma}_i\in H^0(\mo_{\p^2}(2))$, $N_{\gamma}$ can be written as follows.
\begin{equation}\label{tqmqf}N_{\gamma}:=\left(\begin{array}{cccc}0&1&0&0\\a^{\gamma}_1&0&x&0\\a^{\gamma}_2&0&0&x\\a^{\gamma}_3&0&y&z\end{array}\right),
\end{equation}
The map $\bar{p}$ in (\ref{tqfgs}) is defined by $\bar{p}(\gamma)=\frac{det(N_{\gamma})}x.$
  
For every $\alpha\in H^0(\mo_{\{x=0\}}(1))$, $N_{\alpha}=N_{\bar{\jmath}(\alpha)}$ is defined by\begin{equation}\label{tqmqft}N_{\alpha}:=\left(\begin{array}{cccc}0&1&0&0\\ \alpha z&0&x&0\\-\alpha y&0&0&x\\0&0&y&z\end{array}\right).
\end{equation}

$\forall(t_1,t_2)\in\bt$, define an element $\gamma\in H^0(Q_f\otimes\mo_{\p^2}(2))$ to be an eigenvector of $\theta_{(t_1,t_2)}^{*}$ with eigenvalue $\lambda$, if $\exists~\underline{a^{\gamma}}\sim\gamma$ such that
\[\theta^{*}_{(t_1,t_2)}N_{\underline{a^{\gamma}}}=Diag(1,\frac1{t_1},\frac1{t_2},1)\cdot  N_{\lambda\underline{a^{\gamma}}}\cdot Diag(1,1,t_1,t_2).\]  

Let $\alpha\in H^0(\mo_{\{x=0\}}(1))$ satisfying $\theta_{(t_1,t_2)}^{*}\alpha=\lambda_{\alpha}\alpha$, then $\bar{\jmath}(\alpha)$ is of eigenvalue $t_1t_2\lambda_{\alpha}$, in other words, $\theta_{(t_1,t_2)}^{*}\cdot\bar{\jmath}=t_1t_2\bar{\jmath}$.  Let $\gamma\in H^0(Q_f\otimes\mo_{\p^2}(2))$ be an eigenvector with eigenvalue $\lambda$, then $\theta_{(t_1,t_2)}^{*}(\bar{p}(\gamma))=\lambda\bar{p}(\gamma)\in H^0(\mo_{\p^2}(3))$.

$H^0(Q_f\otimes\mo_{\p^2}(2))$ can be decomposed into a direct sum of 10 eigen-subspaces of $\theta^{*}_{(t_1,t_2)}$.  There are 8 one-dimensional eigen-subspaces spanned by $\gamma^{(i,j)}$ with $\bar{p}(\gamma^{(i,j)})=x^{3-i-j}y^iz^j$ of eigenvalues $t_1^it_2^j,$ for $i,j\geq0$, $i+j\leq3$ and $(i,j)$ not equal to $(1,2)$ or $(2,1)$.  These subspaces provide 8 $\bt$-fixed points in $\p(H^0(Q_f\otimes\mo_{\p^2}(2)))$ not contained in $\p(H^0(\mo_{\{x=0\}}(1)))$.  There are 2 two-dimensional eigen-subspaces spanned by $\{\bar{\jmath}(y|_{\{x=0\}}),\gamma^{(2,1)}\}$ and $\{\bar{\jmath}(z|_{\{x=0\}}),\gamma^{(1,2)}\}$ of eigenvalues $t^2_1t_2$ and $t^2_2t_1$ respectively (notice that $\theta_{(t_1,t_2)}^{*}\cdot\bar{\jmath}=t_1t_2\bar{\jmath}$).  These subspaces provide 2 $\bt$-fixed projective lines $\p^1$ in $\p(H^0(Q_f\otimes\mo_{\p^2}(2)))$, either of which intersects $\p(H^0(\mo_{\{x=0\}}(1)))$ at one point, hence we get 2 affine lines $\mathbb{A}^1$ for fixed $[\sigma_f]$.  The same holds for $Q_f$ with torsion $\mo_{\{y=0\}}(-1)$ or $\mo_{\{z=0\}}(-1)$.  Hence we in total have 24 points and 6 affine lines. 
\end{proof}
\begin{proof}[Proof of Theorem \ref{flfour}]Combine Lemma \ref{mtfour}, Lemma \ref{wfour} and Lemma \ref{tfour}.\end{proof}

\section{$\bt$-fixed points in $M(5,1)$ and $M(5,2)$.}
Up to isomorphism $M(5,1)$ and $M(5,2)$ are the only two moduli spaces with $d=5$ such that there is no strictly semistable sheaves.  In this section we prove the following  theorem.
\begin{thm}\label{mtfi}Both $M(5,1)$ and $M(5,2)$ admit affine pavings.  Moreover, 

1, the $\bt$-fixed loci of $M(5,1)$ can be decomposed into the union of 1545 points,  144 affine lines $\mathbb{A}^1$ and 6 affine planes $\mathbb{A}^2$;

2, the $\bt$-fixed loci of $M(5,2)$ can be decomposed into the union of 1506 points,  186 affine lines $\mathbb{A}^1$ and 3 affine planes $\mathbb{A}^2$.

\end{thm}

\begin{flushleft}{\textbf{$\lozenge$ Computation for $[M(5,1)]$}}\end{flushleft}
$M(5,1)$ can be stratified into the following three $\bt$-invariant strata.
\begin{eqnarray} & M_1:=&\widetilde{W}^5=\{[(E,f)]\in M(5,1)|E\simeq\mo_{\p^2}\oplus\mmon^{\oplus 4}\};\nonumber\\& M_2:=&\{[(E,f)]\in M(5,1)|E\simeq\mo_{\p^2}^{\oplus 2}\oplus\mmon^{\oplus 2}\oplus\mo_{\p^2}(-2)\};\nonumber\\& M_3:=&\{[(E,f)]\in M(5,1)|E\simeq\mone\oplus\mo_{\p^2}\oplus\mmon\oplus\mo_{\p^2}(-2)^{\oplus 2}\}.\nonumber
\end{eqnarray}

For a pair $(E,f)\in M_3$, $f$ can be represented by the following matrix
\begin{equation}\label{romfi}\left(\begin{array}{ccccc}0&1&0&0&0\\0&0&1&0&0\\0&0&0&1&0\\a_1&0&0&0&b_1\\ a_2 &0&0&0&b_2\end{array}\right),
\end{equation}
where $b_i\in H^0(\mone)$ and $a_i\in H^0(\mo_{\p^2}(4))$.  The stability condition for such $(E,f)$ is equivalent to $kb_1\neq k'b_2$ for any $(k,k')\in\bc^2-\{0\}$.  
\begin{lemma}\label{mhfive}The $\bt$-fixed locus of $M_3$ consists of 60 points.
\end{lemma}
\begin{proof}The proof is analogous to Lemma \ref{mtfour}.  $M_3$ is isomorphic to a projective bundle over $Hilb^{[1]}(\p^2)$ with fibers isomorphic to $\p(H^0(I_1(5)))\simeq \p^{19}$. Every $\bt$-fixed point in $M_2$ corresponds to a $\bt$-fixed point on $\p^2$ together with a $\bt$-fixed curve of degree 5 passing it.  Hence the lemma.
\end{proof}
We follow the notations in \cite{myself} and stratify $M_2$ into two $\bt$-invariant strata as follows.
\begin{eqnarray} & M_2^s:=&\{[(E,f)]\in M_2|f|_{\mo_{\p^2}^{\oplus 2}\otimes\mmon} is~surjective ~onto ~\mmon^{\oplus 2}\};\nonumber\\& M_2^c:=&M_2-M_2^s.\nonumber
\end{eqnarray}

For a pair $(E,f)\in M_2^s$, $f$ can be represented by the following matrix
\begin{equation}\label{romfit}\left(\begin{array}{ccccc}0&0&1&0&0\\0&0&0&1&0\\0&0&0&0&1\\b_1&b_2&0&0&0\\ a_1 &a_2&0&0&0\end{array}\right),
\end{equation}
where $b_i\in H^0(\mo_{\p^2}(2))$ and $a_i\in H^0(\mo_{\p^2}(3))$.  Such $(E,f)$ is stable if and only if $kb_1\neq k'b_2$ for any $(k,k')\in\bc^2-\{0\}$.  
\begin{lemma}\label{mtfive}The $\bt$-fixed locus of $M_2^s$ consists of 201 points and 27 affine lines.
\end{lemma}
\begin{proof}We have the following diagram
\begin{equation}\label{fficd}\xymatrix{ 0\ar[r]&\mo_{\p^2}(-2)\ar[r]^{(b_1,b_2)}&\mo_{\p^2}^{\oplus 2}\ar[r]^{f_r}& R_f\ar[r]&0\\
&&\mo_{\p^2}(-3)\ar[u]^{(a_1,a_2)}\ar[ru]_{\omega_f:=f_r\circ(a_1,a_2)}&&}.\end{equation}
Since $kb_1\neq k'b_2,\forall(k,k')\in\bc^2-\{0\}$,  the isomorphism classes of $R_f$ 1-1 correspond to points  in $Gr(2,h^0(\mo_{\p^2}(2))).$  $\bt$-fixed points in $M_2^s$ correspond to $\bt$-fixed $R_f$, which give $\bt$-fixed points in $Gr(2,h^0(\mo_{\p^2}(2)))$, and $\bt$-fixed $[\omega_f]\in\p(H^0(R_f\otimes\mo_{\p^2}(3)))-\p(H^0(R_f^t\otimes\mo_{\p^2}(3)))$ with $R_f^t$ the torsion of $R_f.$

There are 6 $\bt$-fixed torsion-free $R_f$ where we can choose $b_i$ $\bt$-fixed and coprime to each other.  For each such $R_f$,  the number of $\bt$-fixed $[\omega_f]$ is $h^0(R_f\otimes\mo_{\p^2}(3))=17$.  Hence we get 102 points.

There are 9 $\bt$-fixed $R_f$ containing torsion where we can choose $b_i$ $\bt$-fixed but not coprime to each other.  Each such $R_f$ lies in the following non-split sequence.
\begin{equation}\label{trfft}0\ra\mo_{H}(-1)\ra R_f\ra I_1\otimes\mo_{\p^2}(1)\ra 0,
\end{equation}
  
Use the argument similar to Lemma \ref{tfour}, we see that the locus of $\bt$-fixed $[\omega_f]$ consists of 11 points and 3 affine lines for each of those 9 $\bt$-fixed $R_f$.  Hence in total we have 201 points and 27 fixed affine lines.  
\end{proof}
  
For a pair $(E,f)\in M_2^c$, $f$ can be represented by the following matrix
\begin{equation}\label{romfitc}\left(\begin{array}{ccccc}b_1&b_2&0&0&0\\0&0&1&0&0\\a_1&a_2&0&b_3&0\\0&0&0&0&1\\ e_1 &e_2&0&a_3&0\end{array}\right),
\end{equation}
where $b_i\in H^0(\mo_{\p^2}(1))$, $a_i\in H^0(\mo_{\p^2}(2))$ and $e_i\in H^0(\mo_{\p^2}(3))$.  $(E,f)$ is stable if and only if $kb_1\neq k'b_2,\forall (k,k')\in\bc^2-\{0\}$ and $k''a_3\neq b\cdot b_3,\forall (k'',b)\in \bc\times H^0(\mone)-\{(0,0)\}$.  
\begin{lemma}\label{mcfive}The $\bt$-fixed locus of $M_2^c$ consists of 462 points and 12 affine lines.\end{lemma}  
\begin{proof}From the proof of Lemma 4.4 in \cite{myself}, we can see that $\bt$-fixed points in $M_2^c$ correspond to $\bt$-fixed isomorphism classes of the triples $(I_1,I_2,[\omega_f])$ where $I_1$ is defined by $(b_1,b_2)$, $I_2$ defined by $(b_3,a_3)$ and $[\omega_f]\in\p(H^0(I_1\otimes I_2\otimes\mo_{\p^2}(5)))-\p(H^0((T\otimes\mo_{\p^2}(5))))$ with $(I_1\otimes I_2)^t$ the torsion of $I_1\otimes I_2$.

Notice that the torsion of $I_1\otimes I_2$ is either zero or isomorphic to a skyscraper sheaf $\mo_{\{x\}}$ with $x$ the single point defined by $I_1$ on $\p^2$.  

There are 3 $\bt$-fixed $I_1$.  For each $I_1$ there are 5 $\bt$-fixed $I_2$ such that $I_1\otimes I_2$ are torsion free and 4 $\bt$-fixed $I_2$ such that $I_1\otimes I_2$ have torsion.  We repeat the argument we did in the proof of Lemma \ref{tfour} and get: for each $I_1\otimes I_2$ torsion free, there are 18 $\bt$-fixed $[\sigma_f]$; while for each $I_1\otimes I_2$ with torsion, the $\bt$-fixed locus of $[\sigma_f]$ consists of 16 points and 1 affine line.  Hence we have in total $3\times5\times18+3\times 4\times 16=462$ points and $3\times4=12$ affine lines.  
\end{proof}

For $(E,f) \in M_1=\widetilde{W}^5$, we get the following diagram
\begin{equation}\label{tqficd}\xymatrix{ 0\ar[r]&\mmon^{\oplus 3}\ar[r]^{~~f_{B^t}}&\mo_{\p^2}^{\oplus 4}\ar[r]^{f_q}& Q_f\ar[r]&0\\
&&\mo_{\p^2}(-2)\ar[u]^{f_{A^t}}\ar[ru]_{\sigma_f=f_q\circ f_{A^t}}&&}.\end{equation}
Moreover $(E,f)$ and $(Q_f,\sigma_f)$ determine each other.

$M_1-W^5$ can be divided into three $\bt$-invariant strata by the form of $Q_f$ as follows, where $T_f$ is the torsion of $Q_f$ and $Q_f^{tf}$ the torsion-free quotient of $Q_f$.
\begin{eqnarray} & \Pi_1:=&\{[(Q_f,\sigma_f)]\in M_1-W^5|T_f\simeq \mo_{2H},Q^{tf}_f\simeq\mone\};\nonumber\\&\Pi_2:=&\{[(Q_f,\sigma_f)]\in M_1-W^5|T_f\simeq \mo_{H}(-1),Q_f^{tf}\simeq I_2\otimes\mo_{\p^2}(2)\};\nonumber\\&\Pi_3:=&\{[(Q_f,\sigma_f)]\in M_1-W^5|T_f\simeq \mo_{H}(-2),Q_f^{tf}\simeq I_1\otimes\mo_{\p^2}(2)\}.\nonumber
\end{eqnarray}

\begin{lemma}\label{pione}The $\bt$-fixed locus of $\Pi_1$ consists of 39 points and 21 affine lines.
\end{lemma}
\begin{proof}Notice that Ext$^i(\mone,\mo_{2H})=0$ for all $i\neq 1$ and Ext$^1(\mone,\mo_{2H})\simeq\bc$, hence $Q_f$ is uniquely determined by the support of $T_f=\mo_{2H}$.  The argument is like Lemma \ref{tfour} but with a bit difference, since $H^0(T_f)\neq0$ which implies that one of the $3\times3$ submatrices of $B$ can be chosen degenerate.  

Assume $T_f\simeq \mo_{\{z^2=0\}}$.  Every element $\gamma$ in $H^0(Q_f\otimes\mo_{\p^2}(2))$ can be assigned to a matrix $N_{\gamma}$ of form (\ref{matrix}), more precisely, write $\gamma\sim\underline{a^{\gamma}}:=(a_1^{\gamma},a_2^{\gamma},a_3^{\gamma},a_4^{\gamma})$ with $a^{\gamma}_i\in H^0(\mo_{\p^2}(2))$, $N_{\gamma}$ can be written as follows.
\begin{equation}\label{tqmqp}N_{\gamma}=N_{\underline{a^{\gamma}}}:=\left(\begin{array}{ccccc}0&1&0&0&0\\a_1^{\gamma}&0&z&0&0\\a_2^{\gamma}&0&-x&z&0\\a_3^{\gamma}&0&0&y&x\\ a_4^{\gamma}&0&y&0&z\end{array}\right),
\end{equation}

We have the following exact sequence.
\begin{equation}\label{tqfgsfiv}\xymatrix@C=0.7cm{0\ar[r]&H^0(\mo_{\{z^2=0\}}\otimes\mo_{\p^2}(2))\ar[r]^{~~\bar{\jmath}} &H^0(Q_f\otimes\mo_{\p^2}(2))\ar[r]^{~~~\bar{p}} &H^0(\mo_{\p^2}(3))\ar[r] &0.}\end{equation}
The map $\bar{p}$ in (\ref{tqfgsfiv}) is defined by $\bar{p}(\gamma)=\frac{det(N_{\gamma})}{z^2}.$
  
For every $\alpha\in H^0(\mo_{\{z^2=0\}}\otimes\mo_{\p^2}(2))$, $N_{\alpha}=N_{\bar{\jmath}(\alpha)}$ is defined by
\begin{equation}\label{tqmqpt}N_{\alpha}:=\left(\begin{array}{ccccc}0&1&0&0&0\\ \alpha &0&z&0&0\\ 0&0&-x&z&0\\0&0&0&y&x\\ 0&0&y&0&z\end{array}\right).
\end{equation}

$\forall(t_1,t_2)\in\bt$, define an element $\gamma\in H^0(Q_f\otimes\mo_{\p^2}(2))$ to be an eigenvector of $\theta_{(t_1,t_2)}^{*}$ with eigenvalue $\lambda$, if $\exists~\underline{\alpha^{\gamma}}\sim\gamma$ such that 
\[\theta^{*}_{(t_1,t_2)}N_{\underline{a^{\gamma}}}=Diag(1,\frac{t_2}{t_1},\frac1{t_1},\frac1{t_2},1)\cdot  N_{\lambda\underline{a^{\gamma}}}\cdot Diag(1,1,t_1,t_1t_2,t_2).\]  

Let $\alpha\in H^0(\mo_{\{z^2=0\}}\otimes\mo_{\p^2}(2))$ satisfying $\theta_{(t_1,t_2)}^{*}\alpha=\lambda_{\alpha}\alpha$, then $\bar{\jmath}(\alpha)$ is of eigenvalue $\frac{t_1}{t_2}\lambda_{\alpha}$, in other words, $\theta_{(t_1,t_2)}^{*}\cdot\bar{\jmath}=\frac{t_1}{t_2}\bar{\jmath}$.  Let $\gamma\in H^0(Q_f\otimes\mo_{\p^2}(2))$ be an eigenvector with eigenvalue $\lambda$, then $\theta_{(t_1,t_2)}^{*}(\bar{p}(\gamma))=\lambda\bar{p}(\gamma)\in H^0(\mo_{\p^2}(3))$.

$H^0(Q_f\otimes\mo_{\p^2}(2))$ can be decomposed into a direct sum of 13 eigen-subspaces of $\theta^{*}_{(t_1,t_2)}$.  There are 8 one-dimensional eigen-subspaces spanned by $\gamma^{(i,j)}$ with $\bar{p}(\gamma^{(i,j)})=x^{3-i-j}y^iz^j$ of eigenvalues $t_1^it_2^j,$ for $i,j\geq0$, $i+j\leq3$ and $(i,j)$ not equal to $(1,0)$ or $(2,0)$.  These subspaces provide 8 $\bt$-fixed points in $\p(H^0(Q_f\otimes\mo_{\p^2}(2)))$ not contained in $\p(H^0(\mo_{\{z^2=0\}}\otimes\mo_{\p^2}(2)))$.  

There are 3 one-dimensional eigen-subspaces spanned by $\bar{\jmath}(x^iy^j|_{z^2=0})$ with$i,j\geq0$, $i+j=2$.  These subspaces are contained in $\p(H^0(\mo_{\{z^2=0\}}\otimes\mo_{\p^2}(2)))$ and don't give any $\sigma_f$.  

There are 2 two-dimensional eigen-subspaces spanned by $\{\bar{\jmath}(yz|_{\{z^2=0\}}),\gamma^{(2,0)}\}$ and $\{\bar{\jmath}(xz|_{\{z^2=0\}}),\gamma^{(1,0)}\}$ of eigenvalues $t^2_1$ and $t_1$ respectively.  These subspaces provide 2 $\bt$-fixed projective lines $\p^1$ in $\p(H^0(Q_f\otimes\mo_{\p^2}(2)))$, either of which intersects $\p(H^0(\mo_{\{z^2=0\}}\otimes\mo_{\p^2}(2)))$ at one point.  

The same holds for $Q_f$ with torsion $\mo_{\{x^2=0\}}$ and $\mo_{\{y^2=0\}}$.  Hence we get in all $24$ points and 6 affine lines.  

If the support of $T_f$ is reduced but not irreducible, one can check these three cases are different from the three cases we just computed, and every eigenvector in $H^0(T_f\otimes\mo_{\p^2}(2))$ gives affine line for $\bt$-fixed $[\sigma_f]$.  We omit the computation.  We get in all 15 points and 15 affine lines.

Therefore in $\Pi_1$ in total $\bt$-fixed locus consists of 39 points and 21 affine lines.
\end{proof} 

Let $(Q_f,\sigma_f)\in \Pi_2$.  By the proof of Lemma 4.6 in \cite{myself} we have the following commutative diagram
\begin{equation}\label{pitcd}\xymatrix@C=0.7cm@R=0.5cm{& &0\ar[d] &0\ar[d]\\
&&\mmon\ar[d]_{\delta_f}\ar[r]^{id}_{=}&\mmon\ar[d]^{(a,b)}&\\
0\ar[r]&\mo_H(-1)\ar[r]\ar[d]_{id}^{=}&\widetilde{Q_f}\ar[d]_{\delta}\ar[r]^{}& \mone\oplus\mo_{\p^2}\ar[r]\ar[d]^g&0~~(*)\\
0\ar[r]&\mo_H(-1)\ar[r]&Q_f\ar[r]^{}\ar[d] &I_2\otimes\mo_{\p^2}(2)\ar[d]\ar[r]&0~~(**)\\& &0 &0}
\end{equation}
where the sequence $(*)$ does not split and $\widetilde{Q_f}$ is the Cartesian product of $Q_f$ and $\mone\oplus\mo_{\p^2}$ over $I_2\otimes\mo_{\p^2}(2)$.  Moreover $\widetilde{Q_f}\simeq\mo_{\p^2}\oplus Q_f^1$ with $Q_f^1$ lying in the following non-splitting sequence
\begin{equation}\label{qfone}0\ra\mo_H\stackrel{\jmath~}{\rightarrow} Q_f^1\otimes\mo_{\p^2}(1)\stackrel{p}{\ra}\mo_{\p^2}(2)\ra0.\end{equation}
Also notice that $Q_f^1$ is unique up to isomorphism for any given $\mo_H(-1)$.

From the diagram (\ref{pitcd}) we can see that $Q_f$ is fixed by $\bt$ iff $\widetilde{Q_f}$ and the map class $[\delta_f]\in\{\delta_f\}/Aut(\widetilde{Q_f})$ are fixed.  By the fact $\widetilde{Q_f}\simeq\mo_{\p^2}\oplus Q_f^1$,  $\widetilde{Q_f}$ is $\bt$-fixed iff $\mo_H(-1)$ is $\bt$-fixed.  Hence $Q_f$ is $\bt$-fixed iff $T_f$ and $[\delta_f]$ are $\bt$-fixed.

\begin{lemma}\label{pitwo}The $\bt$-fixed locus of $\Pi_2$ consists of 264 points, 81 affine lines and 6 affine planes. 
\end{lemma}
\begin{proof}
Denote by $\bar{\jmath}$ and $\bar{p}$ the maps on the global sections induced by $\jmath$ and $p$ in (\ref{qfone}) respectively.  Assume $T_f\simeq\mo_{\{x=0\}}(-1)$.  Let $\delta_f^1\in H^0(Q_f^1\otimes\mo_{\p^2}(1))$ be the restriction of map $\delta_f$ to the direct summand $Q_f^1$ in $\widetilde{Q_f}$.  Then by similar argument to Lemma \ref{tfour}, $H^0(Q^1_f\otimes\mo_{\p^2}(1))$ can be decomposed into a direct sum of 6 eigen-subspaces of $\theta^{*}_{(t_1,t_2)}$.  There are 5 one-dimensional eigen-subspaces spanned by $\gamma^{ij}$ with $\bar{p}(\gamma^{(i,j)})=x^{2-i-j}y^iz^j$ of eigenvalues $t_1^it_2^j,$ for $i,j\geq0$, $i+j\leq2$ and $(i,j)\neq(1,1)$.  These subspaces provide 5 $\bt$-fixed points in $\p(H^0(Q^1_f(1)))$ not contained in $\p(H^0(\mo_{\{x=0\}}))$.  There are 1 two-dimensional eigen-subspaces spanned by $\{\bar{\jmath}(1|_{\{x=0\}}),\gamma^{(1,1)}\}$ of eigenvalue $t_1t_2$ (notice that $\theta^{*}\bar{\jmath}=t_1t_2\bar{\jmath}$).  This subspace provides a $\bt$-fixed projective lines $\p^1$ in $\p(H^0(Q^1_f\otimes\mo_{\p^2}(1)))$ intersecting $\p(H^0(\mo_{\{x=0\}}))$ at one point.  Hence we know that $\bt$-fixed $[\delta_f^1]$ form 5 points and 1 affine line.  

The restriction of $\delta_f$ to $\mo_{\p^2}$ gives $b\in H^0(\mo_{\p^2}(1))$ while $\bar{p}\circ \delta_f^1$ gives $a\in H^0(\mo_{\p^2}(2))$, where $(a,b)$ is defined in (\ref{pitcd}).  We have $g.c.d.(a,b)=1$.  For $[\delta_f^1]=[\gamma^{(0,0)}]$ ($[\gamma^{(2,0)}]$, $[\gamma^{(0,2)}]$ resp.), $a\sim x^2$ ($y^2$, $z^2$ resp.), then there are two choice of $b$ $\bt$-fixed up to scalars: $b\sim y$ or $z$ ($b\sim x$ or $z$, $b\sim x$ or $y$ resp.);  if $[\delta_f^1]$ lies on the affine line or equal to $[\gamma^{(1,0)}]$ or $[\gamma^{(0,1)}]$, then there is only one choice of $b$ for each $a$.  Hence we have $\bt$-fixed $[\delta_f]$ with $T_f\simeq\mo_{\{x=0\}}(-1)$ form 8 points and 1 affine line, and in total $\bt$-fixed $Q_f$ in $\Pi_2$ form 24 points and 3 affine lines.  

For each $\bt$-fixed $Q_f$ all the $\bt$-fixed $[\omega_f]$ form 11 points and 2 affine lines.  Hence in total we have the $\bt$-fixed locus of $\Pi_2$ consists of $24 \times 11=264$ points, $3\times 11+24\times 2=81$ affine lines $\mathbb{A}^1$ and $3\times2=6$ affine planes $\mathbb{A}^2$.
\end{proof}

\begin{lemma}\label{pithree}The $\bt$-fixed locus of $\Pi_3$ consists of 39 points and 3 affine lines.
\end{lemma}  
\begin{proof}
For $(Q_f,\sigma_f)\in\Pi_3$, we have the following two commutative diagram (see the diagram (6.15) and (6.21) in \cite{myself}).
 \begin{equation}\label{tqcdffi}\xymatrix@C=1cm@R=0.5cm{& &0\ar[d] &0\ar[d]\\
0\ar[r]&\mmon^{\oplus 3}\ar[r]^{~~~\imath}\ar[d]^{=}_{id}&K \ar[d]_{j}\ar[r]&T_f\ar[r]\ar[d]&0\\
0\ar[r]&\mmon^{\oplus 3}\ar[r]^{~~~f_{B^t}}&\mo_{\p^2}^{\oplus 4}\ar[d]_{f_{tq}}\ar[r]^{f_q}& Q_f\ar[r]\ar[d]&0\\
& &Q^{tf}_f\ar[r]^{=}_{id}\ar[d] &Q_f^{tf}\ar[d]\\& &0 &0}
\end{equation} 

\begin{equation}\label{cdpth}\xymatrix@C=0.7cm@R=0.5cm{&0\ar[d] &0\ar[d]&\\
0\ar[r]&K\ar[d]\ar[r]^{j}&\mo_{\p^2}^{\oplus 4}\ar[d]^{}\ar[r]^{f_{tq}}&Q^{tf}_f\ar[r]\ar[d]_{id}^{=}&0\\
0\ar[r]&G\ar[r]\ar[d]_{\tau}&\mo_{\p^2}^{\oplus 5}\ar[d]^{\tilde{\tau}}\ar[r]^{g}& Q^{tf}_f\ar[r]&0,\\
&\mo_{\p^2}\ar[d]\ar[r]_{id}^{=}&\mo_{\p^2}\ar[d] &&\\&0 &0 &}
\end{equation}

$T_f\simeq \mo_H(-2)$ hence Hom$(\mmon, T_f)=0$.  The inclusion $\imath$ in (\ref{tqcdffi}) is unique up to isomorphisms of $\mmon^{\oplus 3}$ for a given $K$.  Hence $f_{B^t}$ is determined by the inclusion $j$ and hence is determined by $f_{tq}$.  Parametrizing $f_{B^t}$ is equivalent to parametrizing the surjective map $f_{tq}$.  Hence $Q_f$ is $\bt$-fixed iff $Q_f^{tf}$ and $[f_{tq}]$ are $\bt$-fixed with $[f_{tq}]\in\{f_{tq}\}/Aut(\mo_{\p^2}^{\oplus4})$.

There are 3 $\bt$-fixed $Q_f^{tf}$.  We assume $Q_f^{tf}\simeq I_{\{[0,0,1]\}}\otimes \mo_{\p^2}(2)$,  then $g$ can be represented by a $5\times 1$ matrix $(x^2,xy,xz,yz,y^2)$ and $\tilde{\tau}$ can be represented by  $\underline{h}:=(h_0,h_1,h_2,h_3,h_4)$ with $h_i\in\bc$.  

Notice that $\forall ~N\in Mat_{5\times4}(\bc)$ and $Rank(N)$=4, $f_{tq}$ can be represented by $(x^2,xy,xz,yz,y^2)N$ iff $\underline{h}N=(0,0,0,0)$.  Hence $f_{tq}$ is $\bt$-fixed iff $\exists ~N_{(t_1,t_2)}\in Mat_{5\times4}(\bc)$ of rank 4 such that $\underline{h}N_{(t_1,t_2)}=\underline{h}\cdot Diag(1,t_1,t_2,t_1t_2,t_2^2)N_{(t_1,t_2)}=(0,0,0,0)$ for any $(t_1,t_2)\in\bt$.  Hence $f_{tq}$ is $\bt$-fixed $\Leftrightarrow$ $\underline{h}$ is linearly dependent to $\underline{h}\cdot Diag(1,t_1,t_2,t_1t_2,t_2^2)$ $\Leftrightarrow$ all except one of $h_i$ are zero.

Also by the proof of Lemma 6.9 in \cite{myself}, we can't have $h_1h_2-h_0h_3=h_1^2-h_0h_4=h_1h_3-h_2h_1=0$.  Hence there is only 1 fixed $f_{tq}$ corresponding to $\underline{h}=(0,1,0,0,0)$ up to isomorphism, and in this case $T_f\simeq\mo_{z=0}(-2)$.  Hence in total we have 3 $\bt$-fixed $Q_f$. 

For each $\bt$-fixed $Q_f$, $\bt$-fixed $[\sigma_f]$ form 13 points and 1 affine line. Hence in total we have 39 points and 3 affine lines.
\end{proof}

\begin{lemma}\label{wfive}The $\bt$-fixed locus in $W^5$ consists of 480 points.
\end{lemma}
\begin{proof}By Proposition \ref{bigopen} we know $\bt$-fixed points in $W^5$ are isolated.  One may collect all the fixed points in $W^5$ one by one to get the number of them, or use a tricky way to get the number: by Theorem 6.1 in \cite{myself} the Euler number of $M(5,1)$ is 1695 and we have already computed the $\bt$-fixed locus of the complement of $W^5$, hence we get the Euler number of $W^5$ is 480 and hence the lemma.
\end{proof}
\begin{proof}[Proof of Theorem \ref{mtfi} for $M(5,1)$]Combine Lemma \ref{mhfive}, Lemma \ref{mtfive}, Lemma \ref{mcfive}, Lemma \ref{pione}, Lemma \ref{pitwo}, Lemma \ref{pithree} and Lemma \ref{wfive}.
\end{proof}

\begin{flushleft}{\textbf{$\lozenge$ Computation for $[M(5,2)]$}}\end{flushleft}

We stratify $M(5,2)$ into three $\bt$-invariant strata defined as follows.
\begin{eqnarray} & M_2:=&\{[(E,f)]\in M(5,1)|E\simeq\mo_{\p^2}^{\oplus 2}\oplus\mmon^{\oplus 3}\};\nonumber\\& M_3:=&\{[(E,f)]\in M(5,1)|E\simeq\mo_{\p^2}^{\oplus 3}\oplus\mmon\oplus\mo_{\p^2}(-2)\};\nonumber\\& M_3':=&\{[(E,f)]\in M(5,1)|E\simeq\mone\oplus\mo_{\p^2}\oplus\mmon^{\oplus 2}\oplus\mo_{\p^2}(-2)\}.\nonumber
\end{eqnarray}

For a pair $(E,f)\in M'_3$, $f$ can be represented by the following matrix
\begin{equation}\label{romfiv}\left(\begin{array}{ccccc}0&1&0&0&0\\0&0&1&0&0\\0&0&0&0&1\\d&0&0&b&0\\  c&0&0&a&0\end{array}\right),
\end{equation}
where $b\in H^0(\mone)$, $a\in H^0(\mo_{\p^2}(2))$, $c\in H^0(\mo_{\p^2}(4))$ and $d\in H^0(\mo_{\p^2}(3))$.  $(E,f)$ is stable if and only if $b$ is prime to $a$.  
\begin{lemma}\label{mptfiv}The $\bt$-fixed locus of $M'_3$ consists of 171 points.
\end{lemma}
\begin{proof}By the proof of Lemma 6.12 in \cite{myself}, $M'_3$ is isomorphic to a projective bundle over $Hilb^{[2]}(\p^2)$ with fibers isomorphic to $\p(H^0(I_2(5)))\simeq \p^{18}$. Hence the lemma.
\end{proof}

For a pair $(E,f)\in M_3$, $f$ can be represented by the following matrix
\begin{equation}\label{mmb}\left(\begin{array}{ccc}B&0&0\\ 0&1 &0\\ 0 &0&1\\ A&0&0 \end{array}\right),\end{equation}
where $A$ is a $1\times 3$ matrix with entries in $H^0(\mo_{\p^2}(3))$ and $B$ a $2\times 3$ matrix with entries in $H^0(\mone)$.  
\begin{lemma}\label{mcfivmth}The $\bt$-fixed locus of $M_3$ consists of 216 points and 9 affine lines.
\end{lemma}
\begin{proof}$M_3$ is very closed to $\widetilde{W^4}\subset M(4,1)$: the submatrix $B$ in (\ref{mmb}) has the same parametrizing space $M_B$ with the submatrix $B$ in (\ref{matrix}) for $\widetilde{W}^4$.  Denote by $Hilb^{[3]}(\p^2)_{D}$ the open subscheme parametrizing 3-points on $\p^2$ not lying on a line of degree 1.   By Proposition 4.5 and Lemma 5.6 in \cite{myself}, $M_B$ can be decomposed into the union of $Hilb^{[3]}(\p^2)_D$ and the linear system $|H|$.  Moreover $M_3$ can be decomposed into the union of a projective bundle over $Hilb^{[3]}(\p^2)_D$ with fiber isomorphic to $\p(H^0(I_3(5)))\simeq\p^{17}$ and a difference of two projective bundles over $|H|$ with fibers isomorphic to $\p^{17}$ and $\p(H^0(\mo_H(2)))\simeq \p^2$ respectively.  Anagolous argument to Lemma \ref{tfour} proves the lemma.
\end{proof}

We stratify $M_2$ into two $\bt$-invariant strata as follows.
\begin{eqnarray} & M_2^s:=\{[(E,f)]\in M_2|f_{rs}:\mmon^{\oplus 2}\stackrel{f|_{\mmon^{\oplus 2}}}{\longrightarrow}E\twoheadrightarrow\mmon^{\oplus 3}~is ~injective\};&\nonumber\\& M_2^c:=M_2-M_2^s.\qquad\qquad\qquad\qquad\qquad\qquad\qquad\qquad\qquad\qquad~~~~~~~~~~~~~~~~~~~~~&\nonumber
\end{eqnarray}

For a pair $(E,f)\in M_2^s$, $f$ can be represented by the following matrix
\[\left(\begin{array}{cccc}0&1&0&0\\ 0&0&1 &0\\ A&0&0&B \end{array}\right),\]
where $A$ is a $3\times 2$ matrix with entries in $H^0(\mo_{\p^2}(2))$ and $B$ a $3\times 1$ matrix with entries in $H^0(\mone)$. 

We stratify $M_2^s$ into two $\bt$-invariant strata as follows.
\begin{eqnarray} && \Xi_1:=\{[(E,f)]\in M_2^s|B\simeq (x,y,z)^t\};\nonumber\\&& \Xi_2:=M_2^s-\Xi_1.\nonumber
\end{eqnarray}
\begin{lemma}\label{msxofiv}The $\bt$-fixed locus of $\Xi_1$ consists of 30 points, 30 affine lines and 3 affine planes.
\end{lemma}
\begin{proof}
If $B\simeq (x,y,z)^t$, then $(E,f)$ always satisfies the stability condition.  We have the following diagram
\begin{equation}\label{cdxio}\xymatrix{ 0\ar[r]&\mmon\ar[r]^{~~(x,y,z)}&\mo_{\p^2}^{\oplus 3}\ar[r]^{f_0}& E_0\ar[r]&0\\
&&\mo_{\p^2}(-2)^{\oplus 2}\ar[u]^{f_{A^t}}\ar[ru]_{\xi_f:=f_0\circ f_{A^t}}&&},\end{equation}
with $E_0$ a rank $2$ bundle.  Isomorphism classes $[\xi_f]$ of $\xi_f$ are parametrized by an open subset of $Gr(2,H^0(E_0(2)))\simeq Gr(2,15)$ defined by $det(f)\neq 0$.

We can write down easily an eigenvector basis of $V:=H^0(E_0(2))$ and see that: $V$ can be decomposed into the direct sum of 12 eigen-subspaces for all $\theta_{(t_1,t_2)}$, 3 of which are of two dimension while the rest 9 are of one dimension.  Hence the $\bt$-fixed locus of $Gr(2,H^0(E_0(2)))$ consists of 39 points, 27 projective lines $\p^1$ and 3 pieces of $\p^1\times\p^1$.  After excluding the points where $det(f)=0$, we get 30 points, 27 affine lines $\mathbb{A}^1$ and 3 pieces of $\p^1\times \p^1-\Delta(\p^1)$ with $\Delta:\p^1\ra\p^1\times\p^1$ the diagonal embedding.  $\p^1\times \p^1-\Delta(\p^1)$ can be decomposed into the union of 1 affine line and 1 affine plane.  Hence the lemma.  
\end{proof}

\begin{lemma}\label{msxtfiv}The $\bt$-fixed locus of $\Xi_2$ consists of 522 points, 99 affine lines.
\end{lemma}
\begin{proof}
For a pair $(E,f)\in \Xi_2$, $f$ can be represented by the following matrix
\begin{equation}\label{romxt}\left(\begin{array}{ccccc}0&0&1&0&0\\0&0&0&1&0\\a_1&a_2&0&0&0\\a_3&a_4&0&0&b_1\\  a_5&a_6&0&0&b_2\end{array}\right),
\end{equation}
where $b_i\in H^0(\mone)$ and $a_i\in H^0(\mo_{\p^2}(2))$.  $(E,f)$ is stable if and only if $kb_1\neq k'b_2,ka_1\neq k'a_2,\forall (k,k')\in\bc-\{0\}$. 
We define two sheaves $R_f$ and $S_f$ by the following two exact sequences.
\begin{equation}\label{xtcsmc}\xymatrix{ 0\ar[r]&\mo_{\p^2}(-1)\ar[r]^{~~(b_1,b_2)}&\mo_{\p^2}^{\oplus 2}\ar[r]^{f_r}& R_f\ar[r]&0}\end{equation}
\begin{equation}\label{scsxt}\xymatrix{ 0\ar[r]&\mo_{\p^2}(-2)\ar[r]^{~~(a_1,a_2)}&\mo_{\p^2}^{\oplus 2}\ar[r]^{f_s}& S_f\ar[r]&0}\end{equation}

$R_f\simeq I_1\otimes\mo_{\p^2}(1)$.  Either $S_f\simeq I_4\otimes\mo_{\p^2}(2)$ or $S_f$ lies in the following exact sequence. 
\begin{equation}\label{esxts}0\ra\mo_H(-1)\ra S_f\ra I_1\otimes\mo_{\p^2}(1)\ra0.
\end{equation} 
Isomorphism classes of $(R_f,S_f)$ are parametrized by $Hilb^{[1]}(\p^2)\times Gr(2,H^0(\mo_{\p^2}(2)))$.  $(E,f)\in \Xi_2$ are parametrized by $(R_f,S_f,[\omega_f])$ with $[\omega_f]\in\p(H^0( R_f\otimes S_f\otimes\mo_{\p^2}(2)))-\p(H^0(T\otimes\mo_{\p^2}(2)))$, where $T$ is the torsion of $R_f\otimes S_f$.  

All $\bt$-fixed $(R_f,S_f)$ form 45 points in $Hilb^{[1]}(\p^2)\times Gr(2,H^0(\mo_{\p^2}(2)))$.  There are 6 $\bt$-fixed torson-free $S_f$ while the rest 9 contain torsion.  If $S_f$ is torsion free, then $R_f\otimes S_f$ has torsion $T$ nonzero if and only if $(a_1,a_2)|_x=0$ with $R_f\simeq I_{\{x\}}\otimes\mo_{\p^2}(1)$, and $T$ is isomorphic to the structure sheaf $\mo_{\{x\}}$.  If $S_f$ lies in (\ref{esxts}), then $R_f\otimes S_f$ always has torsion.  If the quotient $I_1\otimes\mo_{\p^2}(1)$ in (\ref{esxts}) $\not\simeq R_f$, then the torsion $T\simeq R_f\otimes\mo_{H}(-1)$ and $h^0(T\otimes\mo_{\p^2}(2))=3$.  If the quotient in (\ref{esxts}) $\simeq R_f\simeq I_{\{x\}}\otimes\mo_{\p^2}(1)$, then $T\simeq(\mo_H(-1)\otimes R_f)\oplus\mo_{\{x\}}$ and $h^0(T\otimes\mo_{\p^2}(2))=4$.

$h^0( R_f\otimes S_f\otimes\mo_{\p^2}(2))=16$.  We do analogous argument to what we did before.  We get 9 $\bt$-fixed $(R_f,S_f)$ such that $T=0$, where $\bt$-fixed $[\omega_f]$ form $9\times 16=144$ points; 9 $\bt$-fixed $(R_f,S_f)$ such that $T\simeq\mo_{\{x\}}$, where $\bt$-fixed $[\omega_f]$ form $9\times 14=126$ points and $9\times 1=9$ affine lines; 18 $\bt$-fixed $(R_f,S_f)$ such that $T\simeq R_f\otimes\mo_{H}(-1)$, where $\bt$-fixed $[\omega_f]$ form $18\times 10=180$ points and $18\times 3=54$ affine lines; and finally 9 $\bt$-fixed $(R_f,S_f)$ such that $T\simeq(\mo_{H}(-1)\otimes R_f)\oplus\mo_{\{x\}}$, 
where $\bt$-fixed $[\omega_f]$ form $9\times 8=72$ points and $9\times 4=36$ affine lines. 

Hence the $\bt$-fixed locus of $\Xi_2$ consists of $144+126+180+72=522$ points and $9+54+36=99$ affine lines.
\end{proof} 

\begin{lemma}\label{mtcfiv}The $\bt$-fixed locus of $M^c_2$ consists of  567 points and 48 affine lines.
\end{lemma}
\begin{proof}
For a pair $(E,f)\in M_2^c$, $f$ can be represented by the following matrix
\[\left(\begin{array}{cccc}b_1&b_2&0&0\\ 0&0&1 &0\\ A_1&A_2&0 &B \end{array}\right),\]
where $b_i\in H^0(\mone)$, $A_i$ is a $3\times 1$ matrix with entries in $H^0(\mo_{\p^2}(2))$ and $B$ a $3\times 2$ matrix with entries in $H^0(\mone)$.  $(E,f)$ is stable, hence $kb_1\neq k'b_2, \forall (k,k')\in\bc^2-\{0\}$ and the parametrizing space $M_B$ of $B$ can be decomposed into the union of $Hilb^{[3]}(\p^2)_D$ and $|H|$, with $Hilb^{[3]}(\p^2)_{D}$ the open subscheme parametrizing 3-points on $\p^2$ not lying on a line of degree 1 (see Proposition 4.5 and Lemma 5.6 in \cite{myself}). 

We write down the following two exact sequences.
\begin{equation}\label{focsmc}\xymatrix{ 0\ar[r]&\mo_{\p^2}(-1)^{\oplus 2}\ar[r]^{~~~~f_{B^t}}&\mo_{\p^2}^{\oplus 3}\ar[r]^{f_r}& R_f\ar[r]&0}\end{equation}
\begin{equation}\label{scsfo}\xymatrix{ 0\ar[r]&\mo_{\p^2}(-1)\ar[r]^{~~(b_1,b_2)}&\mo_{\p^2}^{\oplus 2}\ar[r]^{f_s}& S_f\ar[r]&0}\end{equation}

$S_f\simeq I_1\otimes\mo_{\p^2}(1).$  Either $R_f\simeq I_3\otimes\mo_{\p^2}(2)$ or $R_f$ lies in the following exact sequence. 
\begin{equation}\label{esfos}0\ra\mo_H(-1)\ra R_f\ra \mone\ra0.
\end{equation} 
Isomorphism classes of $(R_f,S_f)$ are parametrized by $M_B\times Hilb^{[1]}(\p^2)$.  

We have the commutative diagrams as follows.
\begin{equation}\label{fomcscd}\xymatrix@C=1.7cm{\mo_{\p^2}(-2)\ar[r]^{A_1^t\oplus A_2^t}&\mo_{\p^2}^{\oplus 6}\ar[d]_{f_r^{\oplus 2}}\ar[r]^{~f_s^{\oplus 2}}& S_f^{\oplus 3}\ar[d]^{id_{S_f}\otimes f_r}\\
 &R_f^{\oplus 2}\ar[r]_{id_{R_f}\otimes f_s~~~} &R_f\otimes S_f.} 
\end{equation}

Isomorphism classes of $(E,f)\in M_2^c$ are parametrized by $(R_f,S_f,[\omega_f])$ with $\omega_f: \mo_{\p^2}(-2)\ra R_f\otimes S_f$ the composed map in (\ref{fomcscd}) and $[\omega_f]$ the class of $\omega_f$ modulo scalars.  Hence $M_2^c$ is an open subset of a projective bundle over $M_B\times Hilb^{[1]}(\p^2)$ with fibers isomorphic to $\p(H^0(R_f\otimes S_f(2)))\simeq \p^{16}$, which is defined by asking the image of $\omega_f$ not contained in the torsion $T$ of $R_f\otimes S_f$, i.e. $det(f)\neq0$. 

There are 3 $\bt$-fixed $S_f$, 10 $\bt$-fixed torsion-free $R_f$ and 3 $\bt$-fixed $R_f$ lying in (\ref{esfos}).

If $R_f$ lies in (\ref{esfos}), then the torsion $T$ of $R_f\otimes S_f$ is isomorphic to $\mo_H(-1)\otimes I_1\otimes\mo_{\p^2}(1)$.  $h^0(T\otimes\mo_{\p^2}(2))=3$.  Hence for all such $(R_f,S_f)$, we have in total $\bt$-fixed $[\omega_f]$ form $11\times 9=99$ points and $3\times 9=27$ affine lines.

If $R_f\simeq I_3\otimes\mo_{\p^2}(2)$.  Assume $S_f\simeq I_{\{[1,0,0]\}}\otimes\mo_{\p^2}(1)$.  The torsion of $R_f\otimes S_f$ is a linear subspace of $\mo_{\{[1,0,0]\}}^{\oplus 2}\simeq\bc^2$ which is the kernel of $B^t|_{[1,0,0]}$.  

There are 4 $\bt$-fixed $R_f$ such that $[1,0,0]\not\in Supp(\mo_{\p^2}(2)/R_f)$.  Then $R_f\otimes I_{\{[1,0,0]\}}$ is torsion free and $\bt$-fixed $[\omega_f]$ form $4\times 17=68$ points in total.  

If $R_f\simeq I_{\{[1,0,0],[0,1,0],[0,0,1]\}}\otimes\mo_{\p^2}(2)$, then $Tor(R_f\otimes I_{\{[1,0,0]\}}) \simeq \mo_{\{[1,0,0]\}}$ and $\bt$-fixed $[\omega_f]$ form $1\times 15=15$ points and $1$ affine line.  

If $R_f$ is one of the 4 $\bt$-fixed sheaves $I_{\{[1,0,0],2[0,1,0]\}}\otimes\mo_{\p^2}(2)$, $I_{\{[1,0,0],2[0,0,1]\}}\otimes\mo_{\p^2}(2)$, $I_{\{2[1,0,0],[0,1,0]\}}\otimes\mo_{\p^2}(2)$ and $I_{\{2[1,0,0],[0,0,1]\}}\otimes\mo_{\p^2}(2)$, then $Tor(R_f\otimes I_{\{[1,0,0]\}}) \simeq \mo_{\{[1,0,0]\}}$ and $\bt$-fixed $[\omega_f]$ form $4\times 15=60$ points and $4$ affine lines.  

Finally, for the last one $\bt$-fixed $R_f\simeq I_{\{3[1,0,0]\}}\otimes\mo_{\p^2}(2)$, $Tor(R_f\otimes S_f) \simeq \mo_{\{[1,0,0]\}}^{\oplus 2}$ and hence $\bt$-fixed $[\omega_f]$ form $1\times 13=13$ points and $2$ affine lines.  

Hence we have in $M_2^c$ $\bt$-fixed points form $99+3\times(68+15+60+13)=567$ points and $27+3\times(1+4+2)=48$ affine lines.
\end{proof}

\begin{proof}[Proof of Theorem \ref{mtfi} for $M(5,2)$]Combine Lemma \ref{mptfiv}, Lemma \ref{mcfivmth}, Lemma \ref{msxtfiv} and Lemma \ref{mtcfiv}, and we get the result by direct computation.
\end{proof}


\end{document}